\documentclass[a4paper,12pt,openany,oneside]{article}
\usepackage[utf8]{inputenc} 
\usepackage{indentfirst} 
\usepackage{fancyhdr} 
\usepackage{graphicx} 
\usepackage{latexsym} 
\usepackage{amsmath,amsthm,amssymb} 
\usepackage[margin=3cm]{geometry} 
\usepackage{xcolor} 
\usepackage{mathrsfs} 
\usepackage[hidelinks]{hyperref} 
\usepackage{lipsum}
\usepackage{verbatim}
\usepackage{stmaryrd}
\usepackage{mathtools} 
\usepackage{tikz-cd} 
\usepackage{csquotes}

\theoremstyle{plain}
\newtheorem{thm}{Theorem}[section]   
 
\newtheorem{lem}[thm]{Lemma} 
\newtheorem{prop}[thm]{Proposition} 
\newtheorem{oss}[thm]{Remark}
\newtheorem{propdef}[thm]{Proposition/Definition}

\theoremstyle{definition} 
\newtheorem{defn}{Definition}[section] 
\theoremstyle{remark} 

\newcommand{\N}{\mathbb{N}}

\newcommand{\R}{\mathbb{R}}

\renewcommand{\Gamma}{\varGamma}

\begin{document}
\title{Diffeological Symplectic Frobenius Reciprocity}
\author{Gabriele BARBIERI${}^{\dagger}$ \& Mauro SPERA${}^{\dagger\dagger}$\footnote{Corresponding author} \\
${}^{\dagger}$ Dipartimento di Matematica e Applicazioni\\
Università di Milano-Bicocca, Via Cozzi 55, 20125 Milano, Italy \\
and \\
Dipartimento di Matematica \enquote{Felice Casorati},\\
Università di Pavia,\\
Via Ferrata 5, 27100 Pavia, Italy\\
e-mail: gabriele.barbieri01@universitadipavia.it\\
\phantom{void} \\
${}^{\dagger\dagger}$ Dipartimento di Matematica e Fisica,\\
Università Cattolica del Sacro Cuore,\\
Via della Garzetta 48, 25133 Brescia, Italy\\
e-mail: mauro.spera@unicatt.it
}

\date{\today}

\maketitle


\begin{abstract}
In this note we prove that the symplectic Frobenius Reciprocity established in the paper \cite{[RZ]} as a set bijection is indeed a diffeological diffeomorphism, as conjectured by its authors. The same holds in the prequantum space context. 
\end{abstract}


\section{Introduction}
The celebrated group theoretical Frobenius Reciprocity (FR for short) has been widely extended via modern representation theory, see e.g. \cite{[K]}. Subsequently, V. Guillemin and S. Sternberg  reformulated the basic ingredients of the representation theoretic approach in symplectic terms (see the foundational \cite{[GS1]}) and proved it for coadjoint orbits of compact Lie groups in \cite{[GS2]}. Recently, T. Ratiu and F. Ziegler (\cite{[RZ]}) succeeded in establishing a symplectic FR as a {\it set-theoretical} bijection between manifolds and conjectured that this bijection is indeed a {\it diffeological} diffeomorphism. The present note aims at proving that this is indeed true.
Diffeology, pioneered by J.M. Souriau \cite{[S1],[S2]} and then mainly developed by P. Iglesias-Zemmour (see e.g. \cite{[IZ]}) proved to be quite effective in analyzing singular spaces such as orbit spaces of group actions - which are also alternatively studied by noncommutative geometry via operator theoretic tools (see e.g. \cite{[C]}) - by virtue of its flexibility and intuitive content mildly departing from that of classical differential geometry.

The outline of the paper is the following. In the first section we collect basic tools from diffeology theory in order to ease readability. Then, in Section 3, we recall the symplectic framework of \cite{[RZ]} together with their "induction in stages" theorem. Subsequently, in Section 4, we provide the proof of our  diffeological FR and this will be done by building on basic constructions in \cite{[W]}. After discussing the prequantum space version of FR in \cite{[RZ]} (Section 5), its diffeological version will be proved in Section 6.
The final section points out possible future research directions.

\section{Tools in diffeology}
 In this preliminary section we are going to  gather together basic diffeological notions and results, referring to \cite{[IZ]} for a thorough treatment.
\begin{defn}
Let $X$ be a non empty set. A \textit{diffeology} on $X$ is a set $\mathcal{D}$ of parametrizations (§1.4 of \cite{[IZ]}) of $X$, such that the following
axioms are satisfied:
\begin{enumerate}
\item The set $\mathcal{D}$ contains the constant parametrizations $P:U\to X, P(r)=x, \forall r\in U$, where $x\in X$, $U\subseteq {\R}^p, p\in {\N}$.
\item Let $P:U\to X, U\subseteq {\R}^p, p\in {\N}$ be a parametrization. If for any $r\in U$
there exists an open neighborhood $V$ of $r$ such that $P|_V\in \mathcal{D}$,
then $P\in \mathcal{D}$.
\item For any $P:U\to X$ in $\mathcal{D}$, for any $V\subseteq {\R}^p, p\in{\N}$ and for any $F\in \mathcal{C}^{\infty}(V,U)$, $P\circ F \in \mathcal{D}$.
\end{enumerate}
The pair $(X, \mathcal{D})$ is called \textit{diffeological space} and the elements of $\mathcal{D}$ are called \textit{plots} of the space $X$.
\end{defn}

\begin{defn}
Let $X, X'$ be diffeological spaces. A map $f: X\to X'$ is said to be \textit{smooth} if, for any plot $P$ of $X$, $f\circ P$ is a plot of $X'$.
A map $f:X\to X'$ is a \textit{diffeomorphism} if it is bijective and both $f$ and $f^{-1}$ are smooth.
\end{defn}

\begin{defn}
Let $(X,\mathcal{D})$ and $(X',\mathcal{D}')$ be two diffeological spaces. The diffeology $\mathcal{D}$ is said to be \textit{finer} than $\mathcal{D}'$ if $\mathcal{D}\subset\mathcal{D}'$. In this case, $\mathcal{D}'$ is said to be \textit{coarser} than $\mathcal{D}$.
\end{defn}

\begin{defn}
Let $f:X\to X'$ be a map, where $X$ is a set and $(X', \mathcal{D}')$ is a diffeological space. The \textit{pullback diffeology} on $X$ is the coarsest diffeology such that $f$ is smooth and it is denoted by $f^*(\mathcal{D}')$. In particular, $P\in f^*(\mathcal{D}')$ if and only if $f\circ P\in \mathcal{D}'$.
\end{defn}

\begin{defn}
Let $f:X\to X'$ be a map between diffeological spaces. The function $f$ is said to be an \textit{induction} if it is injective and $f^*(\mathcal{D}')=\mathcal{D}$, i.e. for any plot $P'$ of $X'$ with values in $f(X)$, $f^{-1}\circ P'$ is a plot of $X$.
\end{defn}

\begin{defn}
Let $X$ be a diffeological space and let $A\subset X$. The \textit{subset diffeology} is the set of the plots $P$ of $X$ with values in $A$.
\end{defn}

\begin{oss}
Let $X$ be a diffeological space and $A\subset X$, endowed with the subset diffeology. The inclusion $j:A\xhookrightarrow{} X$ is an induction ({§1.36 of \cite{[IZ]}}).
\end{oss}

\begin{defn}
Let $f:X\to X'$ be a map between a diffeological space $(X,\mathcal{D})$ and a set $X'$. There exists a finest diffeology on $X'$ such that $f$ is smooth; it is called \textit{pushforward diffeology} and it is denoted by $f_*(\mathcal{D})$. A parametrization $P':U\to X'$ is a plot of $f_*(\mathcal{D})$ if and only if for any $r\in U$ there exists an open neighbourhood $V$ of $r$ such that $P'|_V$ is constant or there exists a plot $Q:V\to X$ of $X$ such that $P'|_V=f\circ Q$.
\end{defn}

\begin{defn}
Let $f:X\to X'$ be a map between diffeological spaces. The function $f$ is said to be a \textit{subduction} if it is surjective and $f_*(\mathcal{D})=\mathcal{D}'$.
\end{defn}

\begin{defn}
Let $(X,\mathcal{D})$ be a diffeological space and let $\mathcal{R}$ be an equivalence relation on $X$. Let $\pi:X\to X/\mathcal{R}$ be the projection onto the quotient. The \textit{quotient diffeology} on $X/\mathcal{R}$ is defined as $\pi_*(\mathcal{D})$.
\end{defn}

\begin{oss}
The projection onto  the quotient, endowed with the quotient diffeology, is a subduction.
\end{oss}

\begin{defn}
Let $\{(X_i, \mathcal{D}_i)\}_{i\in I}$ be a collection of diffeological spaces indexed by some $I$. The \textit{product diffeology} on $\Pi_{i\in {I}}X_i$ is the coarsest diffeology $\mathcal{D}$ such that the projections $\pi_i: \Pi_{j\in I}X_j\to X_i$ are smooth. The pair $(\Pi_{i\in I}X_i, \mathcal{D})$ is called \textit{diffeological product}.
\end{defn}

\begin{oss}
Under the assumptions of the previous definition, the projections $\pi_i$ are subductions ({§1.56 of \cite{[IZ]}}).
\end{oss}

\begin{defn}
Let $X, X'$ be two diffeological spaces and let $\mathcal{C}^{\infty}(X,X')$ be the set of smooth maps between $X$ and $X'$, $f\in \mathcal{C}^{\infty}(X,X'), x\in X$ . Let $ev(f,x)=f(x)$ be the evaluation map. A \textit{functional diffeology} is any diffeology defined on $\mathcal{C}^{\infty}(X,X')$ such that $ev$ is smooth. The \textit{standard functional diffeology} is the coarsest  functional diffeology defined on $\mathcal{C}^{\infty}(X,X')$.
\end{defn}

\begin{defn}
Let $V$ be a vector space over a field $\mathbb{K}$. A diffeology on $V$ is a \textit{vector space diffeology} if the sum of two vectors and the multiplication of a vector by a scalar are smooth.
\end{defn}

\begin{propdef}
Let $V$ be a vector space over a field $\mathbb{K}$. There exists a finest vector space diffeology on $V$ and it is generated by a family of parametrizations $P$ such that
$$P:r\mapsto \sum_{i\in I} \lambda_i(r)v^i,$$ where $I$ is finite, $v^i\in V$ and the $\lambda_i$ are smooth parametrizations of $\mathbb{K}$, defined on the domain of $P$. Equivalently, the plots of this diffeology are parametrizations $P:U\to V$ such that for any $r_0\in U$ there exists an open neighborhood $W\subseteq U$ of $r_0$, a family of smooth parametrizations $\lambda_{i}:W\to \mathbb{K}$ and a family of vectors $v^i\in V$ indexed by a finite set of indices $I$ such that:
$$P|_W:r\mapsto \sum_{i\in I} \lambda_i(r)v^i.$$

This diffeology is called the {\rm fine} diffeology.
\end{propdef}
\begin{proof}
See {§3.7 of \cite{[IZ]}}.
\end{proof}

\begin{defn}
An $n$-dimensional manifold $M$ can be regarded as a diffeological space whose diffeology is generated by local charts (§4.1 of \cite{[IZ]}).
\end{defn}

\begin{defn}
Let $X$ be a diffeological space. A \textit{differential $k$-form} on $X$ is a map $\alpha$ that associates to each plot $P$ of $X$ a differential $k$-form $\alpha(P)$ defined on the domain of $P$, such that, for any smooth parametrization $F$ in the domain of $P$,  the following holds:
$$\alpha(P\circ F)=F^*(\alpha(P)).$$
\end{defn}

\begin{propdef}
Let $X$ and $X'$ be two diffeological spaces, $f:X\to X'$ a smooth map and $\alpha'$ a $k$-form on $X'$. Then there exists a differential $k$-form on $X$, denoted by $f^*(\alpha')$, such that $$f^*(\alpha') (P) =\alpha'(f\circ P),$$ for any $P$ plot of $X$. The $k$-form $f^*(\alpha')$ is called \textit{pullback} of $\alpha'$ by $f$.
\end{propdef}
\begin{proof}
See §6.32 of \cite{[IZ]}.
\end{proof}

\begin{prop}
Let $X$ and $X'$ be two diffeological spaces.
Let $\pi:X\to X'$ be a subduction and let $\alpha$ be a differential $k$-form on $X$.
Then $\alpha$ is the pullback of a $k$-form $\beta$ defined on $X'$ if and only
if, for any $P,Q$ plots of $X$ such that $\pi\circ P=\pi\circ Q$, $\alpha(P)=\alpha(Q)$.
\end{prop}
\begin{proof}
See {§6.38 of \cite{[IZ]}}.
\end{proof}

\begin{defn}
A \textit{diffeological group} $G$ is a group endowed with a compatible diffeology such that multiplication and inversion are smooth maps.
\end{defn}

\begin{defn}
Let $X$ be a diffeological space. Being a subset of $\mathcal{C}^{\infty}(X,X)$, $\mathrm{Diff}(X)$ inherits a functional diffeology. It can also be endowed with the structure of diffeological group: it is the coarsest group diffeology such that the evaluation map is smooth. A parametrization $P$ of $\mathrm{Diff}(X)$ is a plot of the \textit{standard diffeology of group of diffeomorphisms} if and only if $r\mapsto P(r)$ and $r\mapsto P^{-1}(r)$ are plots for the functional diffeology defined on $\mathcal{C}^{\infty}(X,X)$ .

\end{defn}

\begin{defn}
Let $X$ be a diffeological space and $G$ a diffeological group. A \textit{smooth action} of $G$ on $X$ is any smooth homomorphism $\rho: G\to \mathrm{Diff}(X)$, where $\mathrm{Diff}(X)$ is endowed with the functional diffeology of group of diffeomorphisms.
\end{defn}

\section{Symplectic induction in stages}
In this Section we shall act within the framework of \cite{[RZ]}, using the same notation therein.
All manifolds and groups involved will be finite dimensional.\par
Let $H\subset G$ be a closed subgroup of a Lie group $G$ and let us consider a Hamiltonian $H$-space $(Y, \omega_Y,\Psi)$. Let $\overline{\omega}\in T^*G$ such that 
$\overline{\omega}(\delta p):=\langle p, \delta q\rangle$, where $\pi:T^*G\to G$ is the canonical projection and $\pi(p)=q, \delta p\in T_p(T^*G)$, $\delta q:=\pi_*(\delta p)\in T_{\pi(p)}G$. The manifold $N:=T^*G\times Y$ can be endowed with the symplectic form $\omega:=d\overline{\omega}+\omega_Y$ and with a $G\times H$-action:
$$(g,h)(p,y):=(gph^{-1}, h(y)),$$ where $h(y)$ is the action of $h\in H$ on $y\in Y$. 
By \cite{[RZ]}, the equivariant moment map associated to this action is
$\phi\times \psi: N \to \mathfrak{g}^*\times \mathfrak{h}^*$, where
$$\phi(p,y)=pq^{-1}$$
$$\psi(p,y)=\Psi(y)-q^{-1}p_{|\mathfrak{h}}.$$
\begin{defn}
The induced manifold is the reduction of $N$ at $0\in \mathfrak{h}^*$:
$$\mathrm{Ind}_H^GY:=N\sslash H=\psi^{-1}(0)/H.$$
\end{defn}
\begin{oss}
\label{oss: action}
Since the action of $H\subset G$ on $T^*G$ is free and proper, it is free and proper on $N$ as well (see Appendix for a proof). This implies that the induced manifold is actually a manifold.
\end{oss}

\begin{thm}
\label{thm: stages} {\rm (\cite{[RZ]})}
Let $H\subset K\subset G$, where $H$ and $K$ are closed subgroups of the Lie group $G$. Then:
$$\mathrm{Ind}_K^G\mathrm{Ind}_H^KY=\mathrm{Ind}_H^GY.$$
\end{thm}

\begin{oss}
The equality of the previous theorem means that there exists a diffeomorphism between the manifolds $\mathrm{Ind}_K^G\mathrm{Ind}_H^KY$ and $\mathrm{Ind}_H^GY.$ If we regard manifolds as diffeological spaces endowed with the manifold diffeology, the above diffeomorphism can be viewed as a diffeomorphism between diffeological spaces.
\end{oss}
\section{Symplectic Frobenius Reciprocity}
In this section we are going to prove our result. First, we recall the symplectic reformulation of the space of intertwiners in classical representation theory (see e.g. \cite{[K]}) given in §6 of \cite{[GS1]}.
\begin{defn}
\label{defn: hom}
$$\mathrm{Hom}_G(X_1,X_2):=(X_1^{-}\times X_2)\sslash G,$$ where $X_1^-$ is the Hamiltonian $G$-space $X_1$ endowed with opposite $2$-form and moment map.
\end{defn}

\begin{center}
\begin{figure}
\begin{tikzcd}
\node (1-1) {M}; && \node (1-2) {N}; \\
 &\node (2-1) {\overline{\psi}^{-1}(0)}; \\
&&\node(3-1) {(\overline{\phi}\times\overline{\psi})^{-1}(0)}; &&&\node(3-2) {{\psi}^{-1}(0)};\\
&\node(4-1) [yshift=1.4cm] {M\sslash H}; &\node (4-2) {\overline{\Phi}^{-1}_{M\sslash H}(0)};\\
\\
&&\node(5-1) {(M\sslash H)\sslash K};&&& \node(5-2) {N\sslash H};\\
\ar["r", from=1-1, to=1-2]
\ar[hookleftarrow, "j_1", from=1-1, to=2-1]
\ar[hookleftarrow, "j", from=2-1, to=3-1]
\ar["s", from=3-1, to=3-2]
\ar[hookleftarrow, "j_3", from=1-2, to=3-2]
\ar["\pi_1", from=2-1, to=4-1]
\ar["\pi", from=3-1, to=4-2]
\ar[hookleftarrow, "j_2", from=4-1, to=4-2]
\ar["\pi_2", from=4-2, to=5-1]
\ar["\pi_3", from=3-2, to=5-2]
\ar["t", from=5-1, to=5-2]
\end{tikzcd}
\caption{Commutative diagram of \cite{[RZ]}, used in the proof of Theorems 2.1 and 3.4 of \cite{[RZ]} and Theorem \ref{thm: Frobenius}.}
\label{fig: diagram}
\end{figure}
\end{center}
We are in a position to state the following:
\begin{thm}[Symplectic Frobenius Reciprocity]
\label{thm: Frobenius}
Let $H$ and $G$ be Lie groups with $H\subset G$. Let $X$ be a Hamiltonian $G$-space and $Y$ a Hamiltonian $H$-space. Then 
$$\mathrm{Hom}_G(X,\mathrm{Ind}_H^GY)=\mathrm{Hom}_H(\mathrm{Res}_H^GX,Y),$$ where the equality denotes a {\rm diffeological} diffeomorphism.
\end{thm}
\begin{proof}
The first part of the proof coincides with the one of Theorem 3.4 of \cite{[RZ]}; we briefly resume the notation and some intermediate results that will be used in the diffeological part of the proof.

Following the usage in \cite{[RZ]}, let us rename $G$ as $K$ for convenience and set $N:=X^-\times Y$, $M:=X^-\times T^*K\times Y$; these definitions allow us to perform a proof similar to the one of the symplectic induction in stages. An $H$-action is defined on N: $$h(x,y)=(h(x),h(y)),$$ 
and the corresponding moment map is $\psi:N\to \mathfrak{h}^*$, 
\begin{equation}
\label{eq: Psi}
\psi(x,y)=\Psi(y)-\Phi(x)|_{\mathfrak{h}};
\end{equation}
a $K\times H$-action is defined on $M$: $$(k,h)(x,\overline{p},y)=(k(x),k\overline{p}h^{-1},h(y)),$$
with moment map $\overline{\phi}\times \overline{\psi}:M\to \mathfrak{k}^*\times \mathfrak{h}^*,$
\begin{equation}
\label{eq: phi}
\overline{\phi}(x,\overline{p},y)=\overline{p}\,\overline{q}^{-1}-\Phi(x)|_{\mathfrak{h}}
\end{equation}
\begin{equation}
\label{eq: psi}
\overline{\psi}(x,\overline{p},y)=\Psi(y)-\overline{q}^{-1}\overline{p}|_{\mathfrak{h}}.
\end{equation}
Let us denote by $j, j_1, j_2, j_3$ the canonical inclusions and by $\pi, \pi_1,\pi_2, \pi_3$ the canonical projections involved in the construction of the reductions. Let 
$\overline{\Phi}_{M\sslash H}:M\sslash H\to \mathfrak{k}^*$ be the moment map for the action of $K$ on the quotient $M\sslash H$.
Then, by Definition \ref{defn: hom}:
$$N\sslash H=(X^-\times Y)\sslash H=\mathrm{Hom}_H(\mathrm{Res}_H^KX,Y);$$
$$(M\sslash H)\sslash K=(X^-\times(T^*K\times Y)\sslash H) \sslash K=$$
$$= (X^-\times \mathrm{Ind}_H^KY)\sslash K=\mathrm{Hom}_K(X, \mathrm{Ind}_H^KY).$$ 
Before giving the diffeological proof, let us recall the definitions and the properties of the maps $r$, $s$ and $t$ of the diagram, which are proved in Theorem 3.4 of \cite{[RZ]}. Set $r:M\to N$ and 
$r(x,\overline{p},y)=(\overline{q}^{-1}(x),y), \overline{p}\in T^*_{\overline{q}}K;$
let $s:(\overline{\phi}\times \overline{\psi})^{-1}(0)\to \psi^{-1}(0)$, defined as the map $r\circ j \circ j_1$, with range restricted to $\psi^{-1}(0)$. Under this restriction it becomes surjective. Moreover, $s$ is equivariant with respect to the $K\times H$ action on $ (\overline{\phi}\times \overline{\psi})^{-1}(0)$ and to the $H$-action on $\psi^{-1}(0)$; furthermore  the fibers of $s$ are $K$-orbits. These facts imply that $s$ descends to a bijection $t: (M\sslash H)\sslash K \to N\sslash H$ such that the diagram is commutative, i.e. $t\circ \pi_2 \circ \pi=\pi_3 \circ s$. These constructions are portrayed in Fig. \ref{fig: diagram}.

Now we can prove that $t$ is a diffeological diffeomorphism, that is, both $t$ and $t^{-1}$ are smooth (in a diffeological sense). The scheme of the proof is similar to the one of Lemma 3.17 and Proposition 3.43 of \cite{[W]}.
Let us prove that $t:(M\sslash H)\sslash K\to N\sslash H$ is smooth: let 
$p:U\to (M\sslash H)\sslash K$ be a plot of $(M\sslash H)\sslash K$, with $U\subseteq {\R}^p, p\in {\N}$. Then, by definition of quotient diffeology, for any $u\in U$ there exists $V\subseteq U$ and $q:V\to \overline{\Phi}^{-1}_{M\sslash H}(0)$, a plot of $\overline{\Phi}^{-1}_{M\sslash H}(0)$, such that $u\in V$ and $p|_V=\pi_2\circ q$. For the same reason, upon further restrictions on $V$, there exists $\rho:V\to (\overline{\phi}\times \overline{\psi})^{-1}(0)$, a plot of $(\overline{\phi}\times \overline{\psi})^{-1}(0)$, such that $q=\pi \circ \rho$. Since the diagram is commutative we get:
$$t\circ p|_V=t\circ \pi_2 \circ \pi \circ \rho=\pi_3\circ s\circ \rho,$$ hence $t$ is smooth if $s$ is smooth. Now, $s$ is smooth if, for any plot $\sigma$ of $(\overline{\phi}\times \overline{\psi})^{-1}(0)$, $s\circ \sigma$ is a plot of $\psi^{-1}(0)$. Since $j$ is an inclusion, it is an injection with respect to the subset diffeology and then
$j\circ \sigma$ is a plot of $\overline{\psi}^{-1}(0)$; moreover, $j_1\circ j\circ \sigma$ is a plot of $M$ and by commutativity of the diagram 
$$r\circ j_1\circ j\circ \sigma= j_3\circ s \circ \sigma.$$ By Proposition 1.31 of \cite{[IZ]},
$s$ is smooth if $r$ is smooth. By Lemma \ref{lem: r smooth} below, $r$ is smooth, then $s$ is smooth and $t$ is smooth.

Now let us prove that $t^{-1}: N\sslash H\to (M\sslash H)\sslash K$ is smooth. Let $p:U\to N\sslash H$ be a plot of $N\sslash H$, with $U\subseteq {\R}^p,\ p\in {\N}$.
By definition of quotient diffeology, for any $u\in U$ there exists $V\subseteq U$ and $q:V\to \psi^{-1}(0)$, a plot of $\psi^{-1}(0)$, such that $u\in V$ and $p|_V=\pi_3\circ q$. By Lemma \ref{lem: s subduction} below, $s$ is a subduction and by proposition 1.48 of \cite{[IZ]} there exists, upon further restrictions of $V$, a plot $\rho: V\to (\overline{\phi}\times \overline{\psi})^{-1}(0)$ of $(\overline{\phi}\times \overline{\psi})^{-1}(0)$ such that $q=s\circ \rho$.  Commutativity of the diagram then yields
$$t^{-1}\circ p|_V=t^{-1}\circ \pi_3\circ q=t^{-1}\circ \pi_3\circ s\circ \rho=\pi_2\circ \pi \circ \rho.$$ whence $t^{-1}$ is smooth, being the projections smooth. Summarizing, the idea of the proof is to use the diffeological properties of the projections and the inclusions to climb up the diagram so that proving that $t$ and $t^{-1}$ are smooth is tantamount to prove that $r$ is smooth and $s$ is a subduction, respectively. \par

Finally, we prove that $t$ preserves the diffeological forms on the quotients. Indeed, $M$ and $N$ can be endowed, as in Theorem 2.1 of \cite{[RZ]}, with a symplectic form, which is also a diffeological form. Passing to quotients,  note that $(M\sslash H)\sslash K$  and $N\sslash H$ are not necessary manifolds under our hypotheses, which is the reason why we cannot assert that $t$ is a diffeomorphism between manifolds. 
However, considering the subsets of a diffeological space and the quotients of a diffeological space, the symplectic reduction can be endowed with a diffeological space structure and a reduced diffeological form. Indeed, let $(M, \omega, \phi)$ be a Hamiltonian $G$-space. Even if $0\in \mathfrak{g}^*$ is not a regular value with respect to $\phi$, $\phi^{-1}(0)\subset M$ can be endowed with the subset diffeology. Moreover, we can consider $\phi^{-1}(0)/G=: M\sslash G$ as a diffeological space, endowed with the quotient diffeology. Let $j: \phi^{-1}(0)\hookrightarrow M$ be the inclusion and $\pi: \phi^{-1}(0)\to \phi^{-1}(0)/G$ the projection  onto  the quotient. Since $\pi$ is a subduction, there exists a diffeological form on the quotient $\omega_{M\sslash G}$ such that $j^*\omega=\pi^*\omega_{M\sslash G}$ if and only if, for any $P, Q$ plots of $\phi^{-1}(0)$ such that $\pi\circ P=\pi\circ Q$, one has $j^*\omega(P)=j^*\omega(Q)$. From the condition on the projection of the plots it follows that $P$ and $Q$ differ only by the action of $G$; being $\omega$ $G$-invariant, we  indeed have $j^*\omega(P)=j^*\omega(Q)$, hence there exists $\omega_{M\sslash G}$ on $M\sslash G$ such that $j^*\omega=\pi^*\omega_{M\sslash G}$. Thus there exists a diffeological form on the reduction, defined in the same way as the reduced symplectic form. Therefore, we can proceed as in the proof of Theorem 2.1 of \cite{[RZ]}: we have $s^*j_3^*\omega_N=j^*j_1^*\omega_M$ and by commutativity of the diagram $\pi^*\pi_2^*t^*\omega_{N\sslash H}=\pi^*\pi_2^*\omega_{(M\sslash H)\sslash K}.$ Since $\pi_2\circ \pi$ is surjective, its pullback is injective and we have $t^*\omega_{N\sslash H}=\omega_{(M\sslash H)\sslash K}$, as diffeological forms.
\end{proof}

\begin{oss}
In the induction in stages situation $M$ and $N$ were defined as $T^*G\times T^*K\times Y$ and $T^*G\times Y$, respectively. The very presence of $T^*G$ and $T^*K$ made the group actions of $H$ and $K$ free and proper (as in Remark \ref{oss: action}), thus paving the way to quotient manifold structures. In the proof of the previous Theorem, the definitions of $M$ and $N$ do not involve the cotangent spaces $T^*G$ and $T^*K$, thus the action is not necessarily free and proper and the quotients are no longer manifolds. 
\end{oss}

The above result strengthens Theorem 3.4 of \cite{[RZ]}, in which the equality was set-theoretical.

We are now going to prove the missing lemmas.

\begin{lem}
\label{lem: r smooth}
The map $r:M\to N$ is smooth.
\end{lem}
\begin{proof}
Let us consider $M$ and $N$ as in the proof of the previous theorem and the corresponding $K\times H$ and $H$-action, respectively. The map
$r:M\to N$ maps $(x,\overline{p},y)\mapsto (\overline{q}^{-1}(x),y)$, where $\overline{p}\in T^*_{\overline{q}}K$. 
Let us consider a plot $P$ of $M$, endowed with the product diffeology; then $P$ can be defined as follows: $P: U\to M$, $U\subset {\R}^p, p\in {\N}$, $P(u)=(P^X(u), P^T(u), P^Y(u))$,  where $P^X, P^T, P^Y$ are plots of $X, T^*K, Y$, respectively. In particular, since $T^*K=\{(k,\alpha(k))|k\in K, \alpha\in \Omega^1(K)\}$, it can also be endowed with the product diffeology; furthermore, the first factor can be equipped with a Lie group diffeology (§7.1 of \cite{[IZ]}), which is both a group and a manifold diffeology, and the second one with a functional diffeology. These considerations allow us to write the plot of $T^*K$ as $P^T(u)=(P^K(u),P^{\Omega}(u)(P^K(u)))$, where  $P^K$ and $P^{\Omega}$ are plots of $K$ and $\Omega^1(K)$, respectively.  The map $r$ is smooth if $r\circ P$ is a plot of $N$:
$$r\circ P (u)=r((P^X(u), P^T(u), P^Y(u)))=(P^K(u)^{-1}(P^X(u)), P^Y(u)).$$  This is ascertained as follows: first, notice that
 $r$ maps identically the third component of $P$ to the second one of the image, so we only have to focus on the mapping of the first two components of $P$ to the first one of the image. The latter is smooth because, being the action of $K$ smooth on $X$ in the usual sense, the action of $K$, viewed as a diffeological group acting on $X$, is also smooth. Hence $r$ is smooth.
\end{proof}

\begin{lem}
\label{lem: s subduction}
The map $s:(\overline{\phi}\times \overline{\psi})^{-1}(0)\to \psi^{-1}(0)$ is a subduction.
\end{lem}
\begin{proof}
We already know from the proof of Theorem 3.4 of \cite{[RZ]} that $s$ is surjective onto $\psi^{-1}(0)$ and from the proof of Theorem \ref{thm: Frobenius}, combined with Lemma \ref{lem: r smooth}, $s$ is also smooth; then, proving that $s$ is a subduction is tantamount to prove the existence, for any plot $P: U\to \psi^{-1}(0)$ of $\psi^{-1}(0)$, of a plot $Q: V\to (\overline{\phi}\times \overline{\psi})^{-1}(0)$ of $(\overline{\phi}\times \overline{\psi})^{-1}(0)$ such that $P|_V=s\circ Q$, where $V\subseteq U$. Let $P: U\to \psi^{-1}(0), U\subseteq {\R}^p, p\in {\N}$ be a plot of 
$$\psi^{-1}(0)=\{(x,y)\in X^{-}\times Y|\Psi(y)=\Phi(x)|_{\mathfrak{h}}\},$$
which can be endowed with the product diffeology and each component with the manifold diffeology.
Suppose that $Q:V\to (\overline{\phi}\times \overline{\psi})^{-1}(0), V\subseteq U$ is a plot of $(\overline{\phi}\times \overline{\psi})^{-1}(0)$ such that $s\circ Q=P|_V$. Since $(\overline{\phi}\times \overline{\psi})^{-1}(0)\subset M$, $Q(v)=(Q^X(v), Q^T(v),Q^Y(v))$, where $Q^X, Q^T, Q^Y$ are plots of $X, T^*K, Y$ and, as before, $$Q^T(v)=(Q^K(v), Q^{\Omega}(v)(Q^K(v)))=:(\overline{q}(v),\overline{p}(v)),$$ where $\overline{p}(v)\in T^*_{q(v)}K$.
Let $r\in V$, then $$s\circ Q(r)=P(r) \iff (\overline{q}(r)^{-1}(Q^X(r)),Q^Y(r))=(P^X(r),P^Y(r)) $$
$$\iff \overline{q}(r)^{-1}(Q^X(r))=P^X(r) \mbox{ and } Q^Y(r)=P^Y(r)$$
$$\iff Q^X(r)=\overline{q}(r)(P^X(r))\mbox{ and } Q^Y(r)=P^Y(r).$$
Now we can evaluate equation (\ref{eq: Psi}) in $P$, and equations (\ref{eq: phi}) and (\ref{eq: psi}) in $Q$. Since the latter are plots of $\psi^{-1}(0)$ and $(\overline{\phi}\times \overline{\psi})^{-1}(0)$, respectively, we find successively:
\begin{equation}
\label{eq: 1Psi}
\Psi(P^Y(r))=\Phi(P^X(r))|_{\mathfrak{h}}
\end{equation}
\begin{equation}
\label{eq: 1Phi}
\overline{p}(r)\overline{q}(r)^{-1}=\Phi(Q^X(r))=\Phi(\overline{q}(r)(P^X(r))=\overline{q}(r)(\Phi(P^X(r)))
\end{equation}
\begin{equation}
\label{eq: 2Psi}
\Psi(P^Y(r))=\Psi(Q^Y(r))=\overline{q}(r)^{-1}\overline{p}|_{\mathfrak{h}}\implies \Phi(P^X(r))|_{\mathfrak{h}}=\overline{q}(r)^{-1}\overline{p}|_{\mathfrak{h}}.
\end{equation}
Then it follows from (\ref{eq: 1Phi}) that 
$$\Phi(P^X(r))=\overline{q}(r)^{-1}\overline{p}(r)\overline{q}(r)^{-1},$$ 
and substituting in (\ref{eq: 2Psi}) we get 
$$\overline{q}(r)^{-1}\overline{p}(r)|_{\mathfrak{h}}\overline{q}(r)^{-1}=\overline{q}(r)^{-1}\overline{p}(r)|_{\mathfrak{h}}.$$ 
Then 
$\overline{p}(r)|_{\mathfrak{h}}\overline{q}(r)^{-1}=\overline{p}(r)|_{\mathfrak{h}}$, leading to
$\overline{q}(r)^{-1}=1_K$. This holds for all $r\in V$, which implies $\overline{p}(r)=\Phi(P^X(r))$ and $P^X(r)=Q^X(r)$.
Hence $$Q(r)=(P^X(r),(1_K,\Phi(P^X(r))), P^Y(r)),$$ and $V$ can be taken equal to $U$. Now we have to prove that the above determined $Q$  is indeed a plot of $(\overline{\phi}\times \overline{\psi})^{-1}(0)$. If we consider the product diffeology on $(\overline{\phi}\times \overline{\psi})^{-1}(0)$, the first and the third component of $Q$ are, by definition, plots of $X$ and $Y$, respectively; hence we are left to prove that $\Phi(P^X(r))$ is a plot of $\mathfrak{k}^*$ with respect to a functional diffeology. If we consider $\mathfrak{k}$ and ${\R}$ endowed with their fine vector space diffeology, 
by Proposition 3.12 of \cite{[IZ]},
$\Phi(P^X(r))$ is a plot of $\mathfrak{k}^*$. Indeed, if we take $F$ an $n$-dimensional vector subspace of $\mathfrak{k}$, $\Phi(P^X(r))|_F$ writes
$$\Phi(P^X(r))|_F=\sum_{i=1}^{n} a_i(P^X(r))v_i^*,$$ where $(v_1,...,v_n)$ is any basis of $F$, $a_i\in \mathcal{C}^{\infty}(X)$ and $P^X$ is a plot of the manifold diffeology of $X$. In particular, since $a_i\in \mathcal{C}^{\infty}(X)$ and $P^X$ is a plot of $X$, the coefficients are smooth parametrizations of ${\R}$, hence $\Phi(P^X(r))$ is a plot of $\mathfrak{k}^*$. In this case, \enquote{smooth} refers both to the usual and to the diffeological case, since the finest diffeology on ${\R}$ coincides with the smooth one by Exercise 66 of \cite{[IZ]}.
This concludes the proof of the lemma and that of Theorem 4.1.
\end{proof}

\section{Prequantum $G$-spaces}
In this section we recall the prequantum $G$-space framework of \cite{[RZ]}, again in view of extending the corresponding FR to the diffeological context.
\begin{defn} {\cite{[S3]}).}
A \textit{prequantum manifold} $\tilde{X}$ is a manifold endowed with a differential 1-form $\overline{\omega}$ such that, setting $\sigma:=d\overline{\omega}$, the following hold:
$$\mathrm{dim}(\mathrm{ker}\sigma)=1$$
$$\mathrm{dim}(\mathrm{ker}\overline{\omega}\cap\mathrm{ker}\sigma)=0.$$
Moreover, there exists a $U(1)$-action on $\tilde{X}$ such that:
$$z(\tilde{x})=\tilde{x}\iff z=1, \forall z\in U(1), \tilde{x}\in \tilde{X},$$
and a vector field $i$ on $\tilde{X}$, called \textit{Reeb vector field}, such that
$$\sigma(i(\tilde{x}))=0, \,\overline{\omega}(i(\tilde{x}))=1, \forall \tilde{x}\in \tilde{X}.$$
\end{defn}

The definitions of moment map, dual space, product and induction can be rephrased in the prequantum framework.

\begin{propdef}
Let $(\tilde{X}, \overline{\omega})$ be a prequantum manifold. Then $(\tilde{X},d\overline{\omega})$ is a presymplectic manifold whose null leaves are the orbits of $U(1)$ and $d\overline{\omega}$ restricts to a symplectic form $\underline{\omega}$ on the quotient $X=\tilde{X}/U(1)$. If a Lie group $G$ acts on $\tilde{X}$ and preserves $\overline{\omega}$, then it commutes with $U(1)$ and the equivariant moment map 
$$\Phi: \tilde{X}\to \mathfrak{g}^*$$
$$\langle \Phi(\tilde{x}), Z\rangle=\overline{\omega}(Z(\tilde{x})), \tilde{x}\in \tilde{X}, Z\in \mathfrak{g}$$
descends to an equivariant moment map $\underline{\Phi}:X\to \mathfrak{g}^*$. Thus $X$ becomes a Hamiltonian $G$-space prequantized by the prequantum $G$-space $(\tilde{X},\overline{\omega})$ (5.2 of \cite{[RZ]}).
\end{propdef}

\begin{defn} 
We define $\tilde{X}^- $ as the $G$-space that coincides with $\tilde{X}$, but it is endowed with opposite 1-form $-\overline{\omega}$ and, consequently, with opposite Reeb field and $U(1)$-action; it prequantizes $(X^-, -\underline{\omega},-\underline{\Phi})$ (5.3 of \cite{[RZ]}).
\end{defn}

\begin{defn}
Let $\tilde{X}_1$ and $\tilde{X}_2$ be prequantum $G$-spaces. Then $\tilde{X}_1\times \tilde{X}_2$, endowed with the diagonal $G$-action, is a $U(1)\times U(1)$-space. Moreover, the orbits of the action of $\Delta:=\{(z^{-1},z)|z\in U(1)\}$ are the characteristic leaves of the sum of the 1-forms of the prequantum spaces: $\overline{\omega}_1+\overline{\omega}_2$. Passing to the quotient we obtain 
$\tilde{X}_1\boxtimes \tilde{X}_2:=(\tilde{X}_1\times \tilde{X}_2)/\Delta$ and this space prequantizes $X_1\times X_2$. Notice that, by the previous definition, the $\Delta$-action on $\tilde{X}_1^-\times \tilde{X}_2$ is $z(\tilde{x}_1,\tilde{x}_2)=(z(\tilde{x}_1),z(\tilde{x}_2))$ (5.4 of \cite{[RZ]}).
\end{defn}

\begin{propdef}
Let $G$ be a Lie group that acts freely and properly on $\tilde{X}$ and let $L:=\Phi^{-1}(0)$. Being $\Phi$ a moment map, the tangent space to the orbit through $\tilde{x}$ verifies $\mathfrak{g}(\tilde{x})\subset \mathrm{Ker}(\overline{\omega}|_L)\cap\mathrm{Ker}(d\overline{\omega}|_L)$. Due to the $G$-invariance of $\overline{\omega}|_L$, it descends to a contact 1-form on the quotient $\tilde{X}\sslash G:=\Phi^{-1}(0)/G$. The latter prequantizes the symplectic reduction $X\sslash G=\underline{\Phi}^{-1}(0)/G$ (5.5 of \cite{[RZ]}).
\end{propdef}

In view of the above definitions it is possible to construct an induced prequantum $G$-space, as in section 6 of \cite{[RZ]}.
Let $H$ be a closed subgroup of the Lie group $G$ and let $(\tilde{Y}, \overline{\omega}_{\tilde{Y}})$ be a prequantum $H$-space with moment map denoted by $\Psi$. Let $\tilde{N}:=T^*G\times \tilde{Y}$ be the prequantum $(G\times H)$-space endowed with the 1-form $\overline{\omega}_{T^*G}+\overline{\omega}_{\tilde{Y}}$ and action $(g,h)(p,\tilde{y}):=(gph^{-1},h(\tilde{y}))$. This action admits an equivariant moment map $\phi\times \psi:\tilde{N}\to \mathfrak{g}\times\mathfrak{h}^*$, where
$$\phi(p,\tilde{y})=pq^{-1}$$
$$\psi(p,\tilde{y})=\Psi(\tilde{y})-q^{-1}p_{|\mathfrak{h}}.$$
Proceeding as in Remark \ref{oss: action}, it is possible to define the induced prequantum manifold
$$\mathrm{Ind}_H^G\tilde{Y}:=\tilde{N}\sslash H=\psi^{-1}(0)/H.$$
The induced prequantum $G$-space $\mathrm{Ind}_H^G\tilde{Y}$ prequantizes the symplectically induced manifold $\mathrm{Ind}_H^G{Y}$.

In the prequantum framework the induction in stages holds:
\begin{thm}
\label{thm: prequantum stages} {\rm (\cite{[RZ]})}
Let $H\subset K\subset G$, where $H$ and $K$ are closed subgroups of the Lie group $G$. Then:
$$\mathrm{Ind}_K^G\mathrm{Ind}_H^K\tilde{Y}=\mathrm{Ind}_H^G\tilde{Y}.$$
\end{thm}

\section{Prequantum Frobenius Reciprocity}
\begin{defn}
\label{def: prequantum intertwiner}
The \textit{intertwiner space} between two prequantum $G$-spaces $\tilde{X}_1$ and $\tilde{X}_2$ is 
$\mathrm{Hom}_G(\tilde{X}_1,\tilde{X}_2):=(\tilde{X}_1\boxtimes \tilde{X}_2)\sslash G$ (8.1 of \cite{[RZ]}).
\end{defn}
We may now state and prove the following:
\begin{thm}[Prequantum Frobenius Reciprocity]
Let $\tilde{X}$ be a prequantum $G$-space and $\tilde{Y}$ a prequantum $H$-space, then
$$\mathrm{Hom}_G({\tilde X},\mathrm{Ind}_H^G{\tilde Y})=\mathrm{Hom}_H(\mathrm{Res}_H^G{\tilde X},{\tilde Y}),$$ where the equality denotes a {\rm diffeological} diffeomorphism.
\end{thm}

\begin{center}
\begin{figure}
\begin{tikzcd}
&\node (1-1) {\tilde{\tilde{M}}:=\tilde{X}^-\times T^*G\times \tilde{Y}}; &&&& \node (1-2) {\tilde{\tilde{N}}:=\tilde{X}^-\times\tilde{Y}}; \\
&\node (2-1) {\tilde{M}:=\tilde{X}^-\boxtimes T^*G\times \tilde{Y}}; &&&& \node (2-2) {\tilde{\tilde{N}}:=\tilde{X}^-\times\tilde{Y}}; \\
&\node(3-1) {M:=X^-\times T^*G\times Y}; &&&&\node(3-2) {N:=X^-\times Y};\\
\ar["\tilde{\tilde{r}}", from=1-1, to=1-2]
\ar[ "\rm{mod}\,\Delta", from=1-1, to=2-1]
\ar["\rm{mod}\,\Delta", from=1-2, to=2-2]
\ar["\tilde{r}", from=2-1, to=2-2]
\ar["r", from=3-1, to=3-2]
\ar["\rm{mod}\, (U(1)\times U(1))/\Delta", from=2-1, to=3-1]
\ar["\rm{mod}\, (U(1)\times U(1))/\Delta", from=2-2, to=3-2]
\end{tikzcd}
\caption{Commutative diagram 8.1 of \cite{[RZ]}.}
\label{fig: diagram prequantum FR}
\end{figure}
\end{center}
\begin{proof}
The first part of the proof coincides with the one of Theorem 8.2 of \cite{[RZ]}. 
Referring to the commutative diagram in Fig.\ref{fig: diagram prequantum FR}, we can define $\tilde{\tilde{r}}(\tilde{x},p,\tilde{y})=(q^{-1}(\tilde{x}),\tilde{y})$, where $p\in T^*_qG$. The map $\tilde{\tilde{r}}$ descends to a map $\tilde{r}$ and to a map $r$, that coincides with the one of Theorem \ref{thm: Frobenius}; indeed, at each level of the diagram in Fig.\ref{fig: diagram prequantum FR} we have a copy of the one in Fig.\ref{fig: diagram}, where we find the tilded versions of the maps $t$ and $s$. By definition \ref{def: prequantum intertwiner}, proving the Theorem is tantamount to show that $\tilde{t}:(\tilde{M}\sslash H)\sslash G\to \tilde{N}\sslash H$ is a diffeological diffeomorphism. By Theorem 8.2 of \cite{[RZ]}, $\tilde{t}$ is a bijection. It is also (diffeologically) smooth together with its inverse, since Theorem \ref{thm: Frobenius} and Lemmas \ref{lem: r smooth} and \ref{lem: s subduction}  also hold in the prequantum case: their proofs can be performed as before, starting from diffeological spaces $\tilde{X}$ and $\tilde{Y}$ in the tilded version of the diagram in Fig.\ref{fig: diagram}. Thus $\tilde{t}$ is a diffeological diffeomorphism.

\end{proof}

\section{Conclusions and outlook} \par
In this paper, a proof of the Frobenius Reciprocity theorem in the symplectic framework  --
together with a prequantum $G$-space version --
enhancing the Ratiu-Ziegler set-theoretical one has been given, resorting to well established techniques in diffeology theory \cite{[S1],[S2],[IZ],[W]}. This is an encouraging signal towards further applications of the theory in various directions, including infinite dimensional contexts like loop space geometry and geometrical hydrodynamics (see e.g. \cite{[B],[AK],[HSS]}).\par
\null
\noindent

\section*{Appendix}\par
\smallskip
Here, for the sake of completeness, we provide a proof of the statement made in Remark 3.1 (known in the literature in different guises): {\it the free and proper action of $H\subset G$ on $T^*G$  is free and proper on $N$ as well.}\par
\begin{proof}
The action of $H$ on $N$ is defined by $h(p,y)=(ph^{-1},h(y))$. Let $h\in H, (p,y)\in N$ such that $h(p,y)=(p,y)$, then $(ph^{-1},h(y))=(p,y)$ and since the right action of $H$ on $T^*G$ is free and $ph^{-1}=p$, we have $h=1_H$, hence the $H$-action is free on $N$. Furthermore, this implies that $\psi$ is a submersion and $\psi^{-1}(0)$ is a submanifold.
Now let us prove that the $H$-action on $N$ is proper as well. This will imply that  $N\sslash H$ is a manifold.
Let $$f: H\times N\to N\times N, (h, (p,y))\mapsto (h(p,y), (p,y)),$$ and let $U\times V$ be a compact subset of $N\times N$; we want to prove that $$f^{-1}(U\times V)=\{(h,(p,y))\in H\times V|h(p,y)\in U\}$$ is compact. Being $U$ and $V$ compact subsets of $N=T^*G\times Y$, their projections on $T^*G$ and $Y$, say $U_1$, $U_2$, $V_1$, $V_2$, respectively, are compact. Let us also consider the $H$-action on $T^*G$ and the map $$\phi: H\times T^*G\to T^*G\times T^*G, (h,p)\mapsto (ph^{-1}, p).$$ Since the $H$-action on $T^*G$ is proper, $\phi^{-1}(U_1\times V_1):=\tilde{H}\times V_1$ is a compact subset of $H\times V_1$ and, in particular, $\tilde{H}$ is compact. Finally, if $h(p,y)\in U_1\times U_2$, where $(p,y)\in V$, then $(ph^{-1}, p)=\phi(h,(p,y))\in U_1\times V_1$, thus
$h\in \tilde{H}$ and $$f^{-1}(U\times V)=\{(h,(p,y))\in \tilde{H}\times V|h(p,y)\in U\}.$$ Being $U$ closed, $f^{-1}(U\times V)$ is a closed subset of the compact set $\tilde{H}\times V$, then it is compact and  the $H$-action is proper on $N$.
\end{proof}

{\bf Acknowledgements.} The first named author is currently benefiting from a three year Ph.D grant (issued by \enquote{Consorzio Milano Bicocca-Pavia}; project: Geometric quantization and ramifications). The second named author's research is supported by UCSC D1-funds.
Both authors thank A.M. Miti for discussions. Their research is carried out within INDAM-GNSAGA's framework.

\vspace{1cm}

\end{document}